\documentclass[11pt]{amsart}

\usepackage{graphics}
\usepackage{color}
\usepackage[a4paper, margin=3cm]{geometry}
\usepackage{array}
\usepackage{amssymb,enumerate}
\usepackage{amsmath}
\usepackage{bbm}           % used for blackboard 1
\usepackage{bm}
\usepackage{eso-pic,graphicx}
\usepackage{tikz}
\usepackage{cite}
\usepackage{esint}
\usepackage[pdftex, bookmarksnumbered, bookmarksopen, colorlinks, citecolor=blue, linkcolor=blue]{hyperref}

\usepackage{graphicx}
\usepackage{bbding}
\usepackage{makecell}
\usepackage{amsmath,amsfonts}
\usepackage{rotating}
\usepackage{amssymb}
\usepackage{verbatim}
\usepackage{rotating}
\usepackage{mathrsfs}
\DeclareSymbolFont{largesymbols}{OMX}{cmex}{m}{n}
\makeatletter
\def\Ddots{\mathinner{\mkern1mu\raise\p@
\vbox{\kern7\p@\hbox{.}}\mkern2mu
\raise4\p@\hbox{.}\mkern2mu\raise7\p@\hbox{.}\mkern1mu}}
\makeatother

\def\XXint#1#2#3{{\setbox0=\hbox{$#1{#2#3}{\int}$}
\vcenter{\hbox{$#2#3$}}\kern-.5\wd0}}

\begin{document}

\newtheorem{hyp}{Hypothesis}

\newtheorem{hyp2}[hyp]{Hypothesis}

\newtheorem{definition}{Definition}
\newtheorem{theorem}[definition]{Theorem}
\newtheorem{proposition}[definition]{Proposition}
\newtheorem{conjecture}[definition]{Conjecture}
\def\theconjecture{\unskip}
\newtheorem{corollary}[definition]{Corollary}
\newtheorem{lemma}[definition]{Lemma}
\newtheorem{claim}[definition]{Claim}
\newtheorem{sublemma}[definition]{Sublemma}
\newtheorem{observation}[definition]{Observation}
\theoremstyle{definition}

\newtheorem{notation}[definition]{Notation}
\newtheorem{remark}[definition]{Remark}
\newtheorem{question}[definition]{Question}

\newtheorem{example}[definition]{Example}
\newtheorem{problem}[definition]{Problem}
\newtheorem{exercise}[definition]{Exercise}
 \newtheorem{thm}{Theorem}
 \newtheorem{cor}[thm]{Corollary}
 \newtheorem{lem}{Lemma}[section]
 \newtheorem{prop}[thm]{Proposition}
 \theoremstyle{definition}
 \newtheorem{dfn}[thm]{Definition}
 \theoremstyle{remark}
 \newtheorem{rem}{Remark}
 \newtheorem{ex}{Example}
 \numberwithin{equation}{section}

\def\C{\mathbb{C}}
\def\R{\mathbb{R}}
\def\Rn{{\mathbb{R}^n}}
\def\Rns{{\mathbb{R}^{n+1}}}
\def\Sn{{{S}^{n-1}}}
\def\M{\mathbb{M}}
\def\N{\mathbb{N}}
\def\Q{{\mathbb{Q}}}
\def\Z{\mathbb{Z}}
\def\X{\mathbb{X}}
\def\Y{\mathbb{Y}}
\def\F{\mathcal{F}}
\def\L{\mathcal{L}}
\def\S{\mathcal{S}}
\def\supp{\operatorname{supp}}
\def\essi{\operatornamewithlimits{ess\,inf}}
\def\esss{\operatornamewithlimits{ess\,sup}}

\numberwithin{equation}{section}
\numberwithin{thm}{section}
\numberwithin{definition}{section}
%\numberwithin{theorem}{section}
\numberwithin{equation}{section}

\def\earrow{{\mathbf e}}
\def\rarrow{{\mathbf r}}
\def\uarrow{{\mathbf u}}
\def\varrow{{\mathbf V}}
\def\tpar{T_{\rm par}}
\def\apar{A_{\rm par}}

\def\reals{{\mathbb R}}
\def\torus{{\mathbb T}}
\def\t{{\mathcal T}}
\def\heis{{\mathbb H}}
\def\integers{{\mathbb Z}}
\def\z{{\mathbb Z}}
\def\naturals{{\mathbb N}}
\def\complex{{\mathbb C}\/}
\def\distance{\operatorname{distance}\,}
\def\support{\operatorname{support}\,}
\def\dist{\operatorname{dist}\,}
\def\Span{\operatorname{span}\,}
\def\degree{\operatorname{degree}\,}
\def\kernel{\operatorname{kernel}\,}
\def\dim{\operatorname{dim}\,}
\def\codim{\operatorname{codim}}
\def\trace{\operatorname{trace\,}}
\def\Span{\operatorname{span}\,}
\def\dimension{\operatorname{dimension}\,}
\def\codimension{\operatorname{codimension}\,}
\def\nullspace{\scriptk}
\def\kernel{\operatorname{Ker}}
\def\ZZ{ {\mathbb Z} }
\def\p{\partial}
\def\rp{{ ^{-1} }}
\def\Re{\operatorname{Re\,} }
\def\Im{\operatorname{Im\,} }
\def\ov{\overline}
\def\eps{\varepsilon}
\def\lt{L^2}
\def\diver{\operatorname{div}}
\def\curl{\operatorname{curl}}
\def\etta{\eta}
\newcommand{\norm}[1]{ \|  #1 \|}
\def\expect{\mathbb E}
\def\bull{$\bullet$\ }

\def\blue{\color{blue}}
\def\red{\color{red}}

\def\xone{x_1}
\def\xtwo{x_2}
\def\xq{x_2+x_1^2}
\newcommand{\abr}[1]{ \langle  #1 \rangle}

\newcommand{\Norm}[1]{ \left\|  #1 \right\| }
\newcommand{\set}[1]{ \left\{ #1 \right\} }
\newcommand{\ifou}{\raisebox{-1ex}{$\check{}$}}
\def\one{\mathbf 1}
\def\whole{\mathbf V}
\newcommand{\modulo}[2]{[#1]_{#2}}
\def \essinf{\mathop{\rm essinf}}
\def\scriptf{{\mathcal F}}
\def\scriptg{{\mathcal G}}
\def\m{{\mathcal M}}
\def\scriptb{{\mathcal B}}
\def\scriptc{{\mathcal C}}
\def\scriptt{{\mathcal T}}
\def\scripti{{\mathcal I}}
\def\scripte{{\mathcal E}}
\def\V{{\mathcal V}}
\def\scriptw{{\mathcal W}}
\def\scriptu{{\mathcal U}}
\def\scriptS{{\mathcal S}}
\def\scripta{{\mathcal A}}
\def\scriptr{{\mathcal R}}
\def\scripto{{\mathcal O}}
\def\scripth{{\mathcal H}}
\def\scriptd{{\mathcal D}}
\def\scriptl{{\mathcal L}}
\def\scriptn{{\mathcal N}}
\def\scriptp{{\mathcal P}}
\def\scriptk{{\mathcal K}}
\def\frakv{{\mathfrak V}}
\def\v{{\mathcal V}}
\def\C{\mathbb{C}}
\def\D{\mathcal{D}}
\def\R{\mathbb{R}}
\def\Rn{{\mathbb{R}^n}}
\def\rn{{\mathbb{R}^n}}
\def\Rm{{\mathbb{R}^{2n}}}
\def\r2n{{\mathbb{R}^{2n}}}
\def\Sn{{{S}^{n-1}}}
\def\bbM{\mathbb{M}}
\def\N{\mathbb{N}}
\def\Q{{\mathcal{Q}}}
\def\Z{\mathbb{Z}}
\def\F{\mathcal{F}}
\def\L{\mathcal{L}}
\def\G{\mathscr{G}}
\def\ch{\operatorname{ch}}
\def\supp{\operatorname{supp}}
\def\dist{\operatorname{dist}}
\def\essi{\operatornamewithlimits{ess\,inf}}
\def\esss{\operatornamewithlimits{ess\,sup}}
\def\dis{\displaystyle}
\def\dsum{\displaystyle\sum}
\def\dint{\displaystyle\int}
\def\dfrac{\displaystyle\frac}
\def\dsup{\displaystyle\sup}
\def\dlim{\displaystyle\lim}
\def\bom{\Omega}
\def\om{\omega}
\begin{comment}
\def\scriptx{{\mathcal X}}
\def\scriptj{{\mathcal J}}
\def\scriptr{{\mathcal R}}
\def\scriptS{{\mathcal S}}
\def\scripta{{\mathcal A}}
\def\scriptk{{\mathcal K}}
\def\scriptp{{\mathcal P}}
\def\frakg{{\mathfrak g}}
\def\frakG{{\mathfrak G}}
\def\boldn{\mathbf N}
\end{comment}

\author[H. Yang]{Heng Yang}
\address{Heng Yang:
	College of Mathematics and System Sciences \\
	Xinjiang University \\
	Urumqi 830017 \\
	China}
\email{yanghengxju@yeah.net}

\author[J. Zhou]{Jiang Zhou$^{*}$}
\address{Jiang Zhou:
	College of Mathematics and System Sciences \\
	Xinjiang University \\
	Urumqi 830017 \\
	China}
\email{zhoujiang@xju.edu.cn}


\thanks{This work was supported by the National Natural Science Foundation of China (No.12061069).
\thanks{$^{*}$ Corresponding author, e-mail address: zhoujiang@xju.edu.cn}}

\date{\today}
\title[ CHARACTERIZATION OF LIPSCHITZ FUNCTIONS VIA COMMUTATORS  ]
{\bf Characterization of Lipschitz functions via commutators of maximal operators on slice spaces}

\begin{abstract}
Let $0 \leq \alpha<n$, $M_{\alpha}$ be the fractional maximal operator, $M^{\sharp}$ be the sharp maximal operator and $b$ be the locally integrable function. Denote by $[b, M_{\alpha}]$ and $[b, M^{\sharp}]$  be the commutators of the fractional maximal operator  $M_{\alpha}$  and the sharp maximal operator $M^{\sharp}$. In this paper, we show some necessary and sufficient conditions for the boundedness of the commutators $[b, M_{\alpha}]$  and  $[b, M^{\sharp}]$  on slice spaces when the function $b$ is the Lipschitz function, by which some new characterizations of the non-negative Lipschitz function are obtained.
\end{abstract}

\maketitle

\section{Introduction and main results}
Let $T$ be the classical singular integral operator and $b$ be the locally integrable function, the commutator $[b, T]$ is defined by
$$
[b,T]f(x)=bT f(x)-T (bf)(x).
$$
The well-known result of Coifman, Rochberg and Weiss\cite{CRW} showed that the commutator $[b,T]$ is bounded on $L^{p}(\mathbb{R}^{n})$ for $1<p<\infty$ if and only if $b \in BMO(\mathbb{R}^{n})$.
The bounded mean oscillation space $BMO(\mathbb{R}^{n})$ was introduced by John and Nirenberg \cite{JN}, which is defined as the set of all locally integrable functions $f$ on $\mathbb{R}^{n}$ such that
$$
\|f\|_{BMO(\mathbb{R}^{n})}:=\sup _{Q} \frac{1}{|Q|} \int_{Q}\left|f(x)-f_{Q}\right| d x<\infty,
$$
where the supremum is taken over all cubes in $\mathbb{R}^{n}$ and  $f_{Q}:=\frac{1}{|Q|} \int_{Q} f(x) d x$.
In 1978, Janson\cite{J} abtained some characterizations of the Lipschitz space $\dot{\Lambda}_{\beta}\left(\mathbb{R}^{n}\right)$ via the commutator $[b, T]$ and proved that $[b, T]$ is bounded from $L^{p}\left(\mathbb{R}^{n}\right)$ to $L^{q}\left(\mathbb{R}^{n}\right)$ if and only if $b \in \dot{\Lambda}_{\beta}\left(\mathbb{R}^{n}\right)(0<$ $\beta<1)$, where $1<p<$ $n / \beta$ and $1 / p-1 / q=\beta / n$ (see also Paluszy\'{n}ski\cite{P}). Later, the commutators have been studied intensively by many authors(see, for example, \cite{HZ,L,S,ZS} ), which plays an important role in studying the solution of partial differential equations.

As usual, a cube $Q \subset \mathbb{R}^{n}$ always means its sides parallel to the coordinate axes. Denote by $|Q|$ the Lebesgue measure of $Q$ and $\chi_{Q}$ the characteristic function of $Q$. For $1\leq p\leq\infty$, we denote by $p^{\prime}$ the conjugate index of $p$, namely, $p^{\prime}=p /(p-1)$. We always denote by  $C$  a positive constant which is independent of the
main parameters, but it may vary from line to line. The symbol  $f \lesssim g$  means that  $f \leq C g$ . If  $f \lesssim g $ and  $g \lesssim f$ , we then write  $f \sim g $.

Let $0 \leq \alpha<n$, for a locally integrable function $f$, the maximal operator $M_{\alpha}$ is given by
$$M_{\alpha}(f)(x)=\sup _{Q \ni x} \frac{1}{|Q|^{1-\alpha/n}} \int_{Q}|f(y)| d y,$$
where the supremum is taken over all cubes $Q \subset \mathbb{R}^{n}$ containing $x$.

When $\alpha=0$, $M_{0}$ is the classical Hardy-Littlewood maximal operator $M$, and $M_{\alpha}$ is the classical fractional maximal operator when $0<\alpha<n$.

The sharp maximal operator $M^{\sharp}$ was introduced by Fefferman and Stein \cite{FS}, which is defined as
$$
M^{\sharp} f(x)=\sup _{Q \ni x} \frac{1}{|Q|} \int_Q\left|f(y)-f_Q\right| d y,
$$
where the supremum is taken over all cubes $Q \subset \mathbb{R}^{n}$ containing $x$.

The maximal commutator of the fractional maximal operator $M_\alpha$ with the locally integrable function $b$ is given by
$$
M_{\alpha, b}(f)(x)=\sup _{Q \ni x} \frac{1}{|Q|^{1-\alpha / n}} \int_Q|b(x)-b(y)||f(y)| d y,
$$
where the supremum is taken over all cubes $Q \subset \mathbb{R}^n$ containing $x$.

The  nonlinear  commutators of the fractional maximal operator $M_\alpha$ and sharp maximal operator $M^{\sharp}$ with the locally integrable function $b$ are defined as
$$
\left[b, M_\alpha\right](f)(x)=b(x) M_\alpha(f)(x)-M_\alpha(b f)(x)
$$
and
$$
\left[b, M^{\sharp}\right](f)(x)=b(x) M^{\sharp}(f)(x)-M^{\sharp}(b f)(x) .
$$

When $\alpha=0$, we simply write by $[b, M]=\left[b, M_0\right]$ and $M_b=M_{0, b}$.
We also remark that the commutators $M_{\alpha, b}$ and $\left[b, M_\alpha\right]$ essentially differ from each other. For example, maximal commutator $M_{\alpha, b}$ is positive and sublinear, but nonlinear commutators $\left[b, M_\alpha\right]$ and $\left[b, M^{\sharp}\right]$ are neither positive nor sublinear.

To state our results, we first present some definitions and notations.
\begin{definition}\label{de1} Let $0<\beta<1$, we say a function $b$ belongs to the Lipschitz space $\dot{\Lambda}_\beta(\mathbb{R}^n)$ if there exists a constant $C$ such that for all $x, y \in \mathbb{R}^n$,
$$
|b(x)-b(y)| \leq C|x-y|^\beta.
$$
The smallest such constant $C$ is called the $\dot{\Lambda}_\beta$ norm of the function $b$ and is denoted by $\|b\|_{\dot{\Lambda}_\beta}$.
\end{definition}

In 2019, Auscher and Mourgoglou \cite{AM} introduced the slice space $(E_{2}^{p})_{t}(\mathbb{R}^{n})$ with $0<t<\infty$ and  $ 1<p<\infty$, they studied the weak solutions of boundary value problems with a $t$-independent elliptic systems in the upper half plane.
Recently, Auscher and Prisuelos-Arribas\cite{AP} obtained the boundedness of some classical operators on the slice space $(E_{r}^{p})_{t}(\mathbb{R}^{n})$  with  $0<t<\infty$  and  $1<p, r<\infty$.

For $0<p<\infty$, the Lebesgue space  $L^{p}(\mathbb{R}^{n})$  is defined as the set of all measurable functions  $f$  on  $\mathbb{R}^{n}$  such that
$$
\|f\|_{L^p(\mathbb{R}^n)}:=\left(\int_{\mathbb{R}^n}|f(x)|^p \mathrm{~d} x\right)^{\frac{1}{p}}<\infty.
$$
\begin{definition} Let $0<t<\infty$ and $1<r, p<\infty$. The slice space $(E_{r}^{p})_{t}(\mathbb{R}^{n})$ is defined as the set of all locally $r$-integrable functions $f$ on $\mathbb{R}^n$ such that
$$
\|f\|_{(E_{r}^{p})_{t}(\mathbb{R}^{n})}:=\left(\int_{\mathbb{R}^n}\left(\frac{1}{|Q(x, t)|}\int_{Q(x, t)}|f(y)|^{r} d y\right)^{\frac{p}{r}} d x\right)^{\frac{1}{p}}<\infty.
$$
\end{definition}

If we take $r=p$, then the slice space $(E_{r}^{p})_{t}(\mathbb{R}^{n})$ is the Lebesgue space  $L^{p}(\mathbb{R}^{n})$.
For a cube $Q$, we denote by $\|f\|_{(E_r^p)_t(Q)}=\|f\chi_{Q}\|_{(E_r^p)_t(\mathbb{R}^n)}$.

For a fixed cube $Q$ and  $0 \leq \alpha< n$, the maximal operator with respect to $Q$ of a function $f$ is given by
$$
M_{\alpha,Q}(f)(x)=\sup _{Q \supseteq Q_0 \ni x} \frac{1}{|Q_0|^{1-\alpha/n}} \int_{Q_0}|f(y)| dy,
$$
where the supremum is taken over all the cubes $Q_0$ with $Q_0 \subseteq Q$ and $Q_0 \ni x$. Moreover, we denote by $M_{Q} = M_{0,Q}$ when $\alpha= 0$.

In 2017, Zhang \cite{Z} showed some characterizations via the
boundedness of the commutator $[b, M]$ on Lebesgue spaces, when the function $b$ belongs to Lipschitz spaces.

\hspace{-1.2em}{\bf Theorem A.}(\cite{Z}). Let $0<\beta<1$ and $b$ be a locally integrable function. If $1<p<n / \beta$ and $1 / q=1 / p-\beta / n$, then the following statements are equivalent:

(1) $b \in \dot{\Lambda}_\beta(\mathbb{R}^n)$ and $b \geq 0$;

(2) $[b, M]$ is bounded from $L^p(\mathbb{R}^n)$ to $L^q(\mathbb{R}^n)$;

(3) there exists a constant $C>0$ such that
$$
\sup _Q \frac{1}{|Q|^{\beta / n}}\left(\frac{1}{|Q|} \int_Q\left|b(x)-M_Q(b)(x)\right|^q \mathrm{~d} x\right)^{1 / q} \leq C .
$$

Next, we recall the result of \cite{YZ}, which showed some characterizations via the
boundedness of the commutator $[b, M]$ on slice spaces, when the function $b$ belongs to Lipschitz spaces.\\
{\bf Theorem B.} (\cite{YZ}).
Let $0<\beta<1$, $0<t<\infty$ and $b$ be a locally integrable function. If $1<p<r<\infty$, $1<q<s<\infty$ and $\beta/n=1/p-1/r=1/q-1/s$,
then the following statements are equivalent:

(1) $b \in \dot{\Lambda}_{\beta}\left(\mathbb{R}^{n}\right)$ and $b \geq 0$.

(2) $[b, M]$ is bounded from $(E_{p}^{q})_{t}(\mathbb{R}^{n})$ to $(E_{r}^{s})_{t}(\mathbb{R}^{n})$.

(3) There exists a constant $C>0$ such that
\begin{equation} \label{eq1}
\sup _{Q} \frac{1}{|Q|^{\beta/n+1/s}} \|b(\cdot)-M_{Q}(b)(\cdot)\|_{(E_{r}^{s})_{t}(Q)}\leq C.\notag
\end{equation}

Our first result can be stated as follows.

\begin{theorem} \label{TH1.1}Let $0<\beta<1$, $0 \leq \alpha<n$, $0 < \alpha+\beta<n$, $0<t<\infty$ and $b$ be a locally integrable function. If $1<p<r<\infty$, $1<q<s<\infty$ and $(\alpha+\beta)/n=1/p-1/r=1/q-1/s$, then the following statements are equivalent:

(1) $b \in \dot{\Lambda}_{\beta}\left(\mathbb{R}^{n}\right)$ and $b \geq 0$.

(2) $[b, M_{\alpha}]$ is bounded from $(E_{p}^{q})_{t}(\mathbb{R}^{n})$ to $(E_{r}^{s})_{t}(\mathbb{R}^{n})$.

(3) There exists a constant $C>0$ such that
\begin{equation} \label{eq1.1}
\sup _{Q} \frac{1}{|Q|^{\beta/n+1/s}} \|b(\cdot)-M_{Q}(b)(\cdot)\|_{(E_{r}^{s})_{t}(Q)}\leq C.\tag{1.1}
\end{equation}

(4) There exists a constant $C>0$ such that
\begin{equation} \label{eq1.2}
\sup _{Q} \frac{1}{|Q|^{1+\beta / n}} \int_{Q}\left|b(x)-M_{Q}(b)(x)\right| d x \leq C.\tag{1.2}
\end{equation}
\end{theorem}

Here is the second result.
\begin{theorem}\label{TH1.2} Let $0<\beta<1$, $0 \leq \alpha<n$, $0 < \alpha+\beta<n$, $0<t<\infty$ and $b$ be a locally integrable function. If $1<p<r<\infty$, $1<q<s<\infty$ and $(\alpha+\beta)/n=1/p-1/r=1/q-1/s$, then the following statements are equivalent:

(1) $b \in \dot{\Lambda}_{\beta}\left(\mathbb{R}^{n}\right)$.

(2) $ M_{\alpha,b}$ is bounded from $(E_{p}^{q})_{t}(\mathbb{R}^{n})$ to $(E_{r}^{s})_{t}(\mathbb{R}^{n})$.

(3) There exists a constant $C>0$ such that

\begin{equation} \label{eq1.3}
\sup _{Q}\frac{1}{|Q|^{\beta/n+1/s}} \left\|b(\cdot)-b_{Q}\right\|_{(E_{r}^{s})_{t}(Q)} \leq C. \tag{1.3}
\end{equation}

(4) There exists a constant $C>0$ such that
\begin{equation} \label{eq1.4}
\sup _{Q} \frac{1}{|Q|^{1+\beta / n}} \int_{Q}\left|b(x)-b_{Q}\right| d x \leq C.\tag{1.4}
\end{equation}
\end{theorem}
Finally, we obtain the following result.

\begin{theorem}\label{TH1.3} Let $0<\beta<1$, $0<t<\infty$ and $b$ be a locally integrable function. If $1<p<r<\infty$, $1<q<s<\infty$ and $\beta/n=1/p-1/r=1/q-1/s$, then the following statements are equivalent:

(1) $b \in \dot{\Lambda}_{\beta}\left(\mathbb{R}^{n}\right)$ and $b \geq 0$.

(2) $\left[b, M^{\sharp}\right]$ is bounded from $(E_{p}^{q})_{t}(\mathbb{R}^{n})$ to $(E_{r}^{s})_{t}(\mathbb{R}^{n})$.

(3) There exists a constant $C>0$ such that
\begin{equation} \label{eq1.5}
\sup _{Q}\frac{1}{|Q|^{\beta/n+1/s}}\left\|b(\cdot)-2 M^{\sharp}\left(b \chi_{Q}\right)(\cdot)\right\|_{(E_{r}^{s})_{t}(Q)} \leq C.\tag{1.5}
\end{equation}

(4) There exists a constant $C>0$ such that
\begin{equation} \label{eq1.6}
\sup _{Q} \frac{1}{|Q|^{1+\beta / n}} \int_{Q}\left|b(x)-2 M^{\sharp}\left(b \chi_{Q}\right)(x)\right| d x \leq C .\tag{1.6}
\end{equation}
\end{theorem}

\section{Preliminaries}
To prove our results, we give some necessary  lemmas in this section. It is well-known that the Lipschitz space $\dot{\Lambda}_\beta\left(\mathbb{R}^n\right)$ coincides with some Morrey-Companato spaces (see \cite{JTW} for example) and can be characterized by mean oscillation as the following lemma, which is due to DeVore and Sharpley \cite{DS} and Paluszy\'{n}ski \cite{P}.
\begin{lemma} \label{le2.1}
Let $0<\beta<1$ and $1 \leq q<\infty$. The space $\dot{\Lambda}_{\beta, q}\left(\mathbb{R}^{n}\right)$ is defined as the set of all locally integrable functions $f$ such that
$$
\begin{aligned}
\|f\|_{\dot{\Lambda}_{\beta, q}}=\sup _{Q} \frac{1}{|Q|^{\beta / n}}\left(\frac{1}{|Q|} \int_{Q}\left|f(x)-f_{Q}\right|^{q}dx\right)^{1 / q}<\infty.
\end{aligned}
$$
Then, for all $0<\beta<1$ and $1 \leq q<\infty, \dot{\Lambda}_{\beta}\left(\mathbb{R}^{n}\right)=\dot{\Lambda}_{\beta, q}\left(\mathbb{R}^{n}\right)$ with equivalent norms.
\end{lemma}
\begin{lemma}\cite{ZW}\label{le2.2}
Let  $0 \leq \alpha<n$, $Q$  be a cube in  $\mathbb{R}^{n}$  and  $f$  be locally integrable. Then
$$M_{\alpha}\left(f \chi_{Q}\right)(x)=M_{\alpha, Q}(f)(x), \text { for all } x \in Q \text {. }$$
\end{lemma}

The following lemma is given by Lu, Wang and Zhou\cite{LWZ}, they obtained that the boundedness of the fractional maximal operator $M_{\alpha}$ on slice spaces.
\begin{lemma}\label{le2.3} Let $0<t<\infty$, $1<p<r<\infty$ and $1<q<s<\infty$ with $\alpha/n=1/p-1/r=1/q-1/s$ for $0<\alpha<n$. If  $f\in(E_{p}^{q})_{t}(\mathbb{R}^{n})$, then
$$\|M_{\alpha} f\|_{(E_{r}^{s})_{t}(\mathbb{R}^{n})} \leq C\|f\|_{(E_{p}^{q})_{t}(\mathbb{R}^{n})},$$
where the positive constant $C$ is independent of $f$ and $t$.
\end{lemma}

\begin{lemma}\cite{LZW}\label{le2.4} Let $0<t<\infty$, $1<p,r<\infty$ and $Q$ be a cube in $\mathbb{R}^{n}$. Then
$$
\|\chi_{Q}\|_{(E_r^p)_t(\mathbb{R}^n)} \sim |Q|^{1 / p},
$$
\end{lemma}

\begin{lemma}\cite{BMR}\label{le2.5}
For any fixed cube $Q$, let $E=\left\{x \in Q: b(x) \leq b_Q\right\}$ and $F=\left\{x \in Q: b(x)>b_Q\right\}$. Then the following equality is true:
$$
\int_E\left|b(x)-b_Q\right|dx=\int_F\left|b(x)-b_Q\right| dx.
$$
\end{lemma}

\section{Proofs of Theorems \ref{TH1.1}-\ref{TH1.3}}

\begin{proof}[Proof of Theorem \ref{TH1.1}]
(1) $\Rightarrow(2)$: Assume $b \in \dot{\Lambda}_{\beta}(\mathbb{R}^{n})$ and $b \geq 0$. For any locally integral function $f$, we have
\begin{align*}
|[b, M_{\alpha}](f)(x)|&=|b(x) M_{\alpha}(f)(x)-M_{\alpha}(b f)(x)|\\
&\leq\sup _{Q \ni x} \frac{1}{|Q|^{1-\alpha/n}} \int_{Q}|b(x)-b(y) \| f(y)| dy\\
&\leq C\|b\|_{\dot{\Lambda}_{\beta}}\sup _{Q \ni x} \frac{1}{|Q|^{1-(\alpha+\beta)/ n}} \int_{Q}| f(y)| dy\\
&\leq C\|b\|_{\dot{\Lambda}_{\beta}} M_{\alpha+\beta}(f)(x).
\end{align*}
By Lemma \ref{le2.3}, we obtain that $[b,  M_{\alpha}]$ is bounded from $(E_{p}^{q})_{t}(\mathbb{R}^{n})$ to $(E_{r}^{s})_{t}(\mathbb{R}^{n})$.

$(2) \Rightarrow(3)$: We divide the proof into two cases based on the scope of $\alpha$.

{\bf Case 1.} Assume $0 <\alpha<n$. For any fixed cube Q,
$$
\begin{aligned}
& \frac{1}{|Q|^{\beta/n+1/s}} \|b(\cdot)-M_{Q}(b)(\cdot)\|_{(E_{r}^{s})_{t}(Q)} \\
& \leq\frac{1}{|Q|^{\beta/n+1/s}} \left\|b(\cdot)-|Q|^{-\alpha / n} M_{\alpha, Q}(b)(\cdot)\right\|_{(E_{r}^{s})_{t}(Q)} \\
& \quad+\frac{1}{|Q|^{\beta/n+1/s}} \left\||Q|^{-\alpha / n} M_{\alpha, Q}(b)(\cdot)-M_{Q}(b)(\cdot)\right\|_{(E_{r}^{s})_{t}(Q)} \\
& :=I+II .
\end{aligned}
$$
For $I$. By the definition of $M_{\alpha, Q}$, we can see
\begin{equation}\label{eq3.1}
M_{\alpha, Q}\left(\chi_{Q}\right)(x)=|Q|^{\alpha / n},  \text { for all } x \in Q. \tag{3.1}
\end{equation}
Using Lemma \ref{le2.2}, for any $x \in Q$, we have
$$
\begin{aligned}
M_{\alpha}\left(\chi_{Q}\right)(x)=M_{\alpha, Q}\left(\chi_{Q}\right)(x)=|Q|^{\alpha / n},
M_{\alpha}\left(b \chi_{Q}\right)(x)=M_{\alpha, Q}(b)(x).
\end{aligned}
$$
Thus, for any $x \in Q$,
$$
\begin{aligned}
b(x)-|Q|^{-\alpha / n} M_{\alpha, Q}(b)(x) & =|Q|^{-\alpha / n}\left(b(x)|Q|^{\alpha / n}-M_{\alpha, Q}(b)(x)\right) \\
& =|Q|^{-\alpha / n}\left(b(x) M_{\alpha}\left(\chi_{Q}\right)(x)-M_{\alpha}\left(b \chi_{Q}\right)(x)\right) \\
& =|Q|^{-\alpha / n}\left[b, M_{\alpha}\right]\left(\chi_{Q}\right)(x) .
\end{aligned}
$$
Since $\left[b, M_{\alpha}\right]$ is bounded from $(E_{p}^{q})_{t}(\mathbb{R}^{n})$ to $(E_{r}^{s})_{t}(\mathbb{R}^{n})$, then by Lemma \ref{le2.4} and noting that $(\alpha+\beta) / n=1/q-1/s$, we have
$$
\begin{aligned}
I & =|Q|^{-\beta / n-1/s} \left\|b(\cdot)-|Q|^{-\alpha / n} M_{\alpha, Q}(b)(\cdot)\right\|_{(E_{r}^{s})_{t}(Q)} \\
& =|Q|^{-(\alpha+\beta) / n-1/s} \left\|\left[b, M_{\alpha}\right]\left(\chi_{Q}\right)(\cdot)\right\|_{(E_{r}^{s})_{t}(Q)} \\
& \leq C|Q|^{-(\alpha+\beta) / n-1/s} \left\|\chi_{Q}\right\|_{(E_{p}^{q})_{t}(\mathbb{R}^{n})} \\
& \leq C.
\end{aligned}
$$
Next, we estimate $II$. Similar to (\ref{eq3.1}), by Lemma 2.3 and noting that
$$
M_{Q}\left(\chi_{Q}\right)(x)=\chi_{Q}(x), \text { for all } x \in Q,
$$
it is easy to see
\begin{equation}\label{eq3.2}
M\left(\chi_{Q}\right)(x)=\chi_{Q}(x) ~\text {and} ~ M\left(b \chi_{Q}\right)(x)=M_{Q}(b)(x), \text{for any } x \in Q.\tag{3.2}
\end{equation}
Then, by (\ref{eq3.1}) and (\ref{eq3.2}), for any $x \in Q$, we obtain
$$
\begin{aligned}
&\left|| Q|^{-\alpha / n} M_{\alpha, Q}(b)(x)-M_{Q}(b)(x) \right|\\
&\leq  |Q|^{-\alpha / n}\left|M_{\alpha}\left(b \chi_{Q}\right)(x)-|b(x)|M_{\alpha}\left(\chi_{Q}\right)(x)\right| \\
&\quad+|Q|^{-\alpha / n}|| b(x)\left|M_{\alpha}\left(\chi_{Q}\right)(x)-M_{\alpha}\left(\chi_{Q}\right)(x) M\left(b \chi_{Q}\right)(x)\right| \\
&= |Q|^{-\alpha / n}\left|M_{\alpha}\left(|b| \chi_{Q}\right)(x)-\right| b(x)\left|M_{\alpha}\left(\chi_{Q}\right)(x)\right| \\
&\quad+|Q|^{-\alpha / n} M_{\alpha}\left(\chi_{Q}\right)(x)|| b(x)\left|M\left(\chi_{Q}\right)(x)-M\left(b \chi_{Q}\right)(x)\right| \\
&= |Q|^{-\alpha / n}\left|[|b|, M_{\alpha}](\chi_{Q})(x)\right|+\left|[|b|, M]\left(\chi_{Q}\right)(x)\right| .
\end{aligned}
$$
Since $\left[b, M_{\alpha}\right]$ is bounded from $(E_{p}^{q})_{t}(\mathbb{R}^{n})$ to $(E_{r}^{s})_{t}(\mathbb{R}^{n})$ and we can see that $b \in \dot{\Lambda}_{\beta}\left(\mathbb{R}^{n}\right)$ implies $|b| \in \dot{\Lambda}_{\beta}\left(\mathbb{R}^{n}\right)$.
By the definitions of $\left[b, M_{\alpha}\right]$ and $M_{\alpha}$, we have, for any $x \in Q$,
$$
\begin{aligned}
\left|\left[|b|, M_{\alpha}\right]\left(\chi_{Q}\right)(x)\right| &  \leq \sup _{Q^{\prime} \ni x} \frac{1}{\left|Q^{\prime}\right|^{1-\alpha / n}} \int_{Q^{\prime}}\left|b(x)-b(y) \| \chi_{Q}(y)\right| d y \\
& \leq\|b\|_{\dot{\Lambda}_{\beta}\left(\mathbb{R}^{n}\right)} \sup _{Q^{\prime} \ni x} \frac{1}{\left|Q^{\prime}\right|^{1-(\alpha+\beta) / n}} \int_{Q^{\prime}}\left|\chi_{Q}(y)\right| d y \\
& \leq\|b\|_{\dot{\Lambda}_{\beta}} M_{\alpha+\beta}\left(\chi_{Q}\right)(x) \\
& =\|b\|_{\dot{\Lambda}_{\beta}}|Q|^{(\alpha+\beta) / n} \chi_{Q}(x).
\end{aligned}
$$
Similarly, we can see
$$
\left|[|b|, M](\chi_{Q})(x)\right| \leq\|b\|_{\dot{\Lambda}_{\beta}}|Q|^{\beta / n} \chi_{Q}(x),  \text { for any } x \in Q .
$$
Thus,  for any $x \in Q$,
$$
\left|| Q|^{-\alpha / n} M_{\alpha, Q}(b)(x)-M_{Q}(b)(x)\right|\leq C\|b\|_{\dot{\Lambda}_{\beta}}| Q|^{\beta / n} \chi_{Q}(x) .
$$
Then, by Lemma \ref{le2.4}, we have
$$
\begin{aligned}
II & =|Q|^{-\beta / n-1/s} \left\||Q|^{-\alpha / n} M_{\alpha, Q}(b)(\cdot)-M_{Q}(b)(\cdot)\right\|_{(E_{r}^{s})_{t}(Q)} \\
& \leq C |Q|^{-1/s}\left\|\chi_{Q}\right\|_{(E_{r}^{s})_{t}(Q)} \\
& \leq C.
\end{aligned}
$$
This gives the desired estimate
$$
|Q|^{-\beta / n-1/s} \left\|b(\cdot)-M_{Q}(b)(\cdot)\right\|_{(E_{r}^{s})_{t}(Q)} \leq C,
$$
which leads us to (\ref{eq1.1}) since $Q$ is arbitrary and the constant $C$ is dependent of $Q$.

{\bf Case 2.} Assume $\alpha=0$. For any fixed cube $Q$ and any $x \in Q$, by (\ref{eq3.2}), we can see
$$
b(x)-M_{Q}(b)(x)=b(x) M\left(\chi_{Q}\right)(x)-M\left(b \chi_{Q}\right)(x)=[b, M]\left(\chi_{Q}\right)(x).
$$
Assume that $[b, M]$ is bounded from $(E_{p}^{q})_{t}(\mathbb{R}^{n})$ to $(E_{r}^{s})_{t}(\mathbb{R}^{n})$ and $\beta / n=1/q-1/s$, then by Lemma \ref{le2.4}, we have
$$
\begin{aligned}
& |Q|^{-\beta / n-1/s} \left\|b(\cdot)-M_{Q}(b)(\cdot)\right\|_{(E_{r}^{s})_{t}(Q)} \\
& =|Q|^{-\beta / n-1/s}\left\|[b, M]\left(\chi_{Q}\right)(\cdot)\right\|_{(E_{r}^{s})_{t}(Q)} \\
&  \leq C|Q|^{-\beta / n-1/s}\left\|\chi_{Q}\right\|_{(E_{r}^{s})_{t}(\mathbb{R}^{n})} \\
& \leq C,
\end{aligned}
$$
which implies (\ref{eq1.1}).

(3) $\Rightarrow$ (4): Assume (\ref{eq1.1}) holds, then for any fixed cube $Q$, by H\"{o}lder's inequality and (\ref{eq1.1}), we can see
$$
\begin{aligned}
&\frac{1}{|Q|^{1+\beta / n}} \int_{Q}\left|b(x)-M_{Q}(b)(x)\right| d x\\
&\leq \frac{C}{|Q|^{1+\beta/n}} \left\|b(\cdot)-M_{Q}(b)(\cdot)\right\|_{(E_{r}^{s})_{t}(Q)} \|\chi_{Q}\|_{(E_{r'}^{s'})_{t}(\mathbb{R}^{n})}\\
&\leq \frac{C}{|Q|^{\beta/n+1/s}} \left\|b(\cdot)-M_{Q}(b)(\cdot)\right\|_{(E_{r}^{s})_{t}(Q)}\\
&\leq C,
\end{aligned}
$$
where the constant $C$ is independent of $Q$. Thus we have (\ref{eq1.2}).

(4) $\Rightarrow$ (1): To prove $b \in \dot{\Lambda}_\beta\left(\mathbb{R}^n\right)$, by Lemma \ref{le2.1}, it suffices to show that there is a constant $C>0$ such that for any fixed cube $Q$,
$$
\frac{1}{|Q|^{1+\beta / n}} \int_Q\left|b(x)-b_Q\right| dx \leq C .
$$
For any fixed cube $Q$, let $E=\left\{x \in Q: b(x) \leq b_Q\right\}$ and $F=\left\{x \in Q: b(x)>b_Q\right\}$.
Since for any $x \in E$, we have $b(x) \leq b_Q \leq M_Q(b)(x)$, then
\begin{equation}\label{eq3.3}
\left|b(x)-b_Q\right| \leq\left|b(x)-M_Q(b)(x)\right| .\tag{3.3}
\end{equation}
By Lemma \ref{le2.5} and (\ref{eq3.3}), we obtain
$$
\begin{aligned}
\frac{1}{|Q|^{1+\beta / n}} \int_Q\left|b(x)-b_Q\right| \mathrm{d} x
& =\frac{2}{|Q|^{1+\beta / n}} \int_E\left|b(x)-b_Q\right| \mathrm{d} x \\
& \leq \frac{2}{|Q|^{1+\beta / n}} \int_E\left|b(x)-M_Q(b)(x)\right| \mathrm{d} x \\
& \leq \frac{2}{|Q|^{1+\beta / n}} \int_Q\left|b(x)-M_Q(b)(x)\right| \mathrm{d} x\\
& \leq C.
\end{aligned}
$$
Thus we obtain $b \in \dot{\Lambda}_\beta\left(\mathbb{R}^n\right)$.
Next, we will prove $b \geq 0$, it suffices to show $b^{-}=0$, where $b^{-}=-\min \{b, 0\}$. Let $b^{+}=|b|-b^{-}$, then $b=b^{+}-b^{-}$. For any fixed cube $Q$ and $x \in Q$, we observe that
$$
0 \leq b^{+}(x) \leq|b(x)| \leq M_Q(b)(x),
$$
then it is easy to see
$$
0 \leq b^{-}(x) \leq M_Q(b)(x)-b^{+}(x)+b^{-}(x)=M_Q(b)(x)-b(x).
$$
Combining with the above estimates and (\ref{eq1.2}), we obtain
$$
\begin{aligned}
\frac{1}{|Q|} \int_Q b^{-}(x) \mathrm{d} x & \leq \frac{1}{|Q|} \int_Q\left|M_Q(b)(x)-b(x)\right| \\
& \leq|Q|^{\beta / n}\left(\frac{1}{|Q|^{1+\beta/n}} \int_Q\left|b(x)-M_Q(b)(x)\right| dx\right) \\
& \leq C|Q|^{\beta / n} .
\end{aligned}
$$
Thus, $b^{-}=0$ follows from Lebesgue's differentiation theorem.

This completes the proof of Theorem \ref{TH1.1}.
\end{proof}

\begin{proof}[Proof of Theorem \ref{TH1.2}] (1) $\Rightarrow$ (2): Assume $b \in \dot{\Lambda}_\beta(\mathbb{R}^n)$. For any fixed cube $Q\subset \mathbb{R}^n$, we have
$$
\begin{aligned}
M_{\alpha,b}(f)(x) & =\sup _{Q \ni x}\frac{1}{|Q|^{1-\alpha / n}} \int_Q|b(x)-b(y) \| f(y)| d y \\
& \leq C\|b\|_{\dot{\Lambda}_\beta(\mathbb{R}^n)} M_{\alpha+\beta} f(x).
\end{aligned}
$$
By Lemma \ref{le2.3}, we obtain that $M_{\alpha,b}$ is bounded from $(E_{p}^{q})_{t}(\mathbb{R}^{n})$ to $(E_{r}^{s})_{t}(\mathbb{R}^{n})$.

$(2) \Rightarrow(3)$: For any fixed cube $Q\subset \mathbb{R}^n$ and all $x \in Q$, we have
$$
\begin{aligned}
\left|b(x)-b_{Q}\right| & \leq \frac{1}{|Q|} \int_{Q}|b(x)-b(y)| d y \\
& =\frac{1}{|Q|^{\alpha / n}} \frac{1}{|Q|^{1-\alpha / n}} \int_{Q}|b(x)-b(y)| \chi_{Q}(y) d y \\
& \leq|Q|^{-\alpha / n} M_{\alpha, b}\left(\chi_{Q}\right)(x) .
\end{aligned}
$$
Since $M_{\alpha,b}$ is bounded from $(E_{p}^{q})_{t}(\mathbb{R}^{n})$ to $(E_{r}^{s})_{t}(\mathbb{R}^{n})$, then by Lemma \ref{le2.4} and noting that $(\alpha+\beta)/n=1/q-1/s$, we obtain
$$
\begin{aligned}
 \frac{1}{|Q|^{\beta/ n+1/s}}\left\|b(\cdot)-b_{Q}\right\|_{E_{r}^{s})_{t}(Q)}
& \leq|Q|^{-(\alpha+\beta) / n-1/s} \left\|M_{\alpha, b}\left(\chi_{Q}\right)(\cdot)\right\|_{(E_{r}^{s})_{t}(Q)} \\
&  \leq C|Q|^{-(\alpha+\beta) / n-1/s} \left\|\chi_{Q}\right\|_{(E_{p}^{q})_{t}\left(\mathbb{R}^{n}\right)} \\
& \leq C,
\end{aligned}
$$
which implies (\ref{eq1.3}) since the cube $Q\subset\mathbb{R}^{n}$ is arbitrary.

(3) $\Rightarrow$ (4): Assume (\ref{eq1.3}) holds, we will prove (\ref{eq1.4}). For any fixed cube $Q$, by H\"{o}lder's inequality and Lemma \ref{le2.4}, it is easy to see
$$
\begin{aligned}
\frac{1}{|Q|^{1+\beta / n}} \int_{Q}\left|b(x)-b_{Q}\right| d x
& \leq \frac{C}{|Q|^{1+\beta / n}} \left\|b(\cdot)-b_{Q}\right\|_{(E_{r}^{s})_{t}(Q)} \|\chi_{Q}\|_{(E_{r'}^{s'})_{t}(\mathbb{R}^{n})}\\
& \leq \frac{C}{|Q|^{\beta/n+1/s}} \left\|b(\cdot)-b_{Q}\right\|_{(E_{r}^{s})_{t}(Q)}\\
& \leq C.
\end{aligned}
$$

(4) $\Rightarrow$ (1): It follows from Lemma \ref{le2.1} directly, thus we omit the details.

The proof of Theorem \ref{TH1.2} is  finished.
\end{proof}
\begin{proof}[Proof of Theorem \ref{TH1.3}]
(1) $\Rightarrow$ (2): Assume $b \in \dot{\Lambda}_{\beta}\left(\mathbb{R}^{n}\right)$ and $b \geq 0$. For any locally integral function $f$, the following estimate was given in \cite{Z2}:
$$
\left|[b, M^{\sharp}] f(x)\right| \leq C\|b\|_{\dot{\Lambda}_{\beta}} M_{\beta}(f)(x).
$$
Then, by Lemma \ref{le2.3}, we obtain that $\left[b, M^{\sharp}\right]$ is bounded from $(E_{p}^{q})_{t}(\mathbb{R}^{n})$ to $(E_{r}^{s})_{t}(\mathbb{R}^{n})$.

$(2) \Rightarrow(3)$: Assume $[b, M^{\sharp}]$ is bounded from $(E_{p}^{q})_{t}(\mathbb{R}^{n})$ to $(E_{r}^{s})_{t}(\mathbb{R}^{n})$, we will prove (\ref{eq1.5}). For any fixed cube $Q$, we have (see \cite{BMR} for details)
$$
M^{\sharp}\left(\chi_{Q}\right)(x)=1 / 2,  \text { for all } x \in Q \text {. }
$$
Then, for all $x \in Q$,
$$
\begin{aligned}
b(x)-2 M^{\sharp}\left(b \chi_{Q}\right)(x)
& =2\left(b(x) M^{\sharp}\left(\chi_{Q}\right)(x)-M^{\sharp}\left(b \chi_{Q}\right)(x)\right) \\
& =2\left[b, M^{\sharp}\right]\left(\chi_{Q}\right)(x) .
\end{aligned}
$$
Since $\left[b, M^{\sharp}\right]$ is bounded from $(E_{p}^{q})_{t}(\mathbb{R}^{n})$ to $(E_{r}^{s})_{t}(\mathbb{R}^{n})$, then  applying Lemma \ref{le2.4} and noting that $\beta/n=1/q-1/s$, we obtain
$$
\begin{aligned}
& |Q|^{-\beta / n-1/s}\left\|b(\cdot)-2 M^{\sharp}\left(b \chi_{Q}\right)(\cdot)\right\|_{(E_{r}^{s})_{t}(Q)} \\
& =2|Q|^{-\beta / n-1/s}\left\|\left[b, M^{\sharp}\right]\left(\chi_{Q}\right)\right\|_{(E_{r}^{s})_{t}(Q)} \\
&  \leq C|Q|^{-\beta / n-1/s} \left\|\chi_{Q}\right\|_{(E_{p}^{q})_{t}\left(\mathbb{R}^{n}\right)} \\
&  \leq C,
\end{aligned}
$$
where the constant $C$ is independent of $Q$. Then we achieve (\ref{eq1.5}).

$(3) \Rightarrow(4)$: Assume (\ref{eq1.5}) holds, we will prove (\ref{eq1.6}). For any fixed cube $Q$, it follows from H\"{o}lder's inequality and (\ref{eq1.5}) that
$$
\begin{aligned}
& \frac{1}{|Q|^{1+\beta / n}} \int_{Q}\left|b(x)-2 M^{\sharp}\left(b \chi_{Q}\right)(x)\right| d x \\
& \leq C|Q|^{-\beta / n-1/s} \left\|b(\cdot)-2 M^{\sharp}\left(b \chi_{Q}\right)(\cdot)\right\|_{(E_{p}^{q})_{t}(Q)} \\
& \leq C,
\end{aligned}
$$
which implies (\ref{eq1.6}) since the constant $C$ is independent of $Q$.

$(4) \Longrightarrow(1)$: We first prove $b \in \dot{\Lambda}_{\beta}\left(\mathbb{R}^{n}\right)$. For any fixed cube $Q$, the following estimate was given in \cite{BMR}:
$$
\frac{1}{|Q|} \int_{Q}\left|b(x)-b_{Q}\right| d x \leq \frac{2}{|Q|} \int_{Q}\left|b(x)-2 M^{\sharp}\left(b \chi_{Q}\right)(x)\right| d x .
$$
Then by (\ref{eq1.6}), we have
$$
\frac{1}{|Q|^{1+\beta / n}} \int_{Q}\left|b(x)-b_{Q}\right| d x \leq \frac{2}{|Q|^{1+\beta / n}} \int_{Q}\left|b(x)-2 M^{\sharp}\left(b \chi_{Q}\right)(x)\right| d x \leq C,
$$
which leads to $b \in \dot{\Lambda}_{\beta}\left(\mathbb{R}^{n}\right)$ by Lemma \ref{le2.1}.

Now, let us prove $b \geq 0$. It suffices to show $b^{-}=0$, where $b^{-}=-\min \{b, 0\}$ and let $b^{+}=|b|-b^{-}$. For any fixed cube $Q$, we have (see \cite{BMR} for details)
$$
\left|b_{Q}\right| \leq 2 M^{\sharp}\left(b \chi_{Q}\right)(x) \text {, for any } x \in Q \text {. }
$$
Then, for all $x \in Q$,
$$
2 M^{\sharp}\left(b \chi_{Q}\right)(x)-b(x) \geq\left|b_{Q}\right|-b(x)=\left|b_{Q}\right|-b^{+}(x)+b^{-}(x).
$$
By (\ref{eq1.6}), we obtain
\begin{equation}\label{eq3.4}
\left|b_{Q}\right|-\frac{1}{|Q|} \int_{Q} b^{+}(x) d x+\frac{1}{|Q|} \int_{Q} b^{-}(x) d x \leq C|Q|^{\beta / n},\tag{3.4}
\end{equation}
where the constant $C$ is independent of $Q$.

Let the side length of $Q$ tends to 0 (then $|Q| \rightarrow 0$ ) with $x \in Q$. By Lebesgue's differentiation theorem, we obtain that the limit of the left-hand side of (\ref{eq3.4}) equals to
$$
|b(x)|-b^{+}(x)+b^{-}(x)=2 b^{-}(x)=2\left|b^{-}(x)\right|.
$$
Moreover, the right-hand side of (\ref{eq3.4}) tends to $0$. Thus, we have $b^{-}=0$.

The proof of Theorem \ref{TH1.3} is completed.
\end{proof}

\end{document}